\documentclass[11pt,a4paper]{article}

\usepackage{geometry}
\usepackage[T1]{fontenc}
\usepackage[utf8]{inputenc}
\usepackage{cite}
\usepackage{graphicx}
\usepackage{amsmath}
\usepackage{tikz}
\usepackage{float}
\usepackage{amsthm}
\usepackage{amssymb}
\usepackage{mathrsfs}
\usepackage{latexsym}
\usepackage{setspace}
\usepackage{enumerate}
\usepackage{dsfont}
\usepackage{chngcntr}
\usepackage{mathtools}
\usepackage{ragged2e}
\usepackage{pstricks-add}
\usepackage{pgf,tikz}
\usetikzlibrary{arrows}
\usepackage{multicol}
\usepackage{url}
\usepackage{bm}
\usepackage{centernot}
\usepackage{hyperref}
\hypersetup{colorlinks}
\usepackage{authblk}
\usepackage{csquotes}
\usepackage{extarrows}
\usepackage{nameref}

\definecolor{Blue}{RGB}{0, 0, 250}
\definecolor{Green}{RGB}{0, 204, 0}

\DeclareMathOperator{\mod1}{\mathord{mod}}
\DeclareMathOperator{\rank}{\mathord{rank}}
\DeclareMathOperator{\dom}{\mathord{dom}}
\DeclareMathOperator{\Ee}{\mathord{E}}
\DeclareMathOperator{\Rr}{\mathrel\mathscr{R}}
\DeclareMathOperator{\Ll}{\mathrel\mathscr{L}}
\DeclareMathOperator{\Hh}{\mathrel\mathscr{H}}
\DeclareMathOperator{\Jj}{\mathrel\mathscr{J}}
\DeclareMathOperator{\Dd}{\mathrel\mathscr{D}}
\DeclareMathOperator{\Kk}{\mathrel\mathcal{K}}
\DeclareMathOperator{\Rrr}{\mathord{R}}
\DeclareMathOperator{\Lll}{\mathord{L}}
\DeclareMathOperator{\Hhh}{\mathord{H}}
\DeclareMathOperator{\Jjj}{\mathord{J}}
\DeclareMathOperator{\Ddd}{\mathord{D}}
\DeclareMathOperator{\Kkk}{\mathord{K}}
\DeclareMathOperator{\Ss}{\mathord{S}}
\DeclareMathOperator{\Tt}{\mathcal{T}}
\DeclareMathOperator{\Pt}{\mathcal{PT}}
\DeclareMathOperator{\id}{\mathord{id}}
\DeclareMathOperator{\Reg}{\mathord{Reg}}
\DeclareMathOperator{\Pp}{\mathord{P}}
\DeclareMathOperator{\codom}{\mathord{codom}}
\DeclareMathOperator{\coker}{\mathord{coker}}
\DeclareMathOperator{\Nu}{\mathord{N_{U}}}
\DeclareMathOperator{\Nl}{\mathord{N_{L}}}
\DeclareMathOperator{\Bd}{\mathord{\mathcal{B}}}
\DeclareMathOperator{\PBd}{\mathord{\mathcal{PB}}}

\newcommand{\nc}{\newcommand}
\nc{\rnc}{\renewcommand}

\makeatletter
\let\orgdescriptionlabel\descriptionlabel
\renewcommand*{\descriptionlabel}[1]{%
  \let\orglabel\label
  \let\label\@gobble
  \phantomsection
  \edef\@currentlabel{#1}%
  \let\label\orglabel
  \orgdescriptionlabel{#1}%
}
\makeatother

\nc{\uvert}[1]{\fill (#1,1.5)circle(.2);}
\rnc{\lvert}[1]{\fill (#1,0)circle(.2);}

\nc{\disppartnx}[3]{{\lower1.0 ex\hbox{\begin{tikzpicture}[scale=0.3]
\foreach \x in {1,...,#1}
{ \uvert{\x}  }
\foreach \y in {1,...,#2}
{ \lvert{\y} }
#3 \end{tikzpicture}}}}

\nc{\darcx}[3]{\draw(#1,0)arc(180:90:#3) (#1+#3,#3)--(#2-#3,#3) (#2-#3,#3) arc(90:0:#3);}
\nc{\uarcx}[3]{\draw(#1,1.5)arc(180:270:#3) (#1+#3,1.5-#3)--(#2-#3,1.5-#3) (#2-#3,1.5-#3) arc(270:360:#3);}

\author{Ivana Đurđev Brković}
\affil{\emph{Mathematical Institute of the Serbian Academy of Sciences and Arts}\\
\emph{Kneza Mihaila 36, 11000 Beograd, Serbia}\\
\href{mailto:ivana.djurdjev@mi.sanu.ac.rs}{ivana.djurdjev@mi.sanu.ac.rs}}
\date{}
\title{Classification of variants of partial Brauer monoids}

\theoremstyle{plain}
 \newtheorem{thm}{Theorem}[section]
 \newtheorem{prop}{Proposition}[section]
 \newtheorem{lem}{Lemma}[section]
 
 \newtheorem{hyp}{Hypothesis}[section]
\theoremstyle{definition}
 \newtheorem{exm}{Example}[section]
 
\theoremstyle{remark}
 \newtheorem{rem}{Remark}[section]

\begin{document}

\maketitle


\begin{abstract}

A \emph{variant} of a semigroup $S$ with respect to an element $a\in S$ is the semigroup $S^{a}=(S,\star_{a})$, where $x\star_{a}y=xay$ for any $x,y\in S$. Here, $a$ is the \emph{sandwich element} of $S^{a}$. In this article, we study variants of the partial Brauer monoid $\mathcal{PB}_{n}$ for $n\in\mathbb{N}$. We give the classification of these variants in the case when the rank of the sandwich element is nonzero.

\vspace{0.2cm}

\vspace{0.2cm}

\end{abstract}


\section{Introduction}

The idea of a sandwich operation is a natural one in semigroup theory. Namely, it arises in relation to Rees matrix semigroups, which are the building blocks of finite semigroups. Sandwich semigroups - semigroups defined by such an operation, first appeared in Lyapin's 1960 monograph \cite{Lyapin}. However, it was not until the 1980's that variants were named and investigated by Hickey \cite{Hickey1,Hickey2}. He used them to provide a natural interpretation of the famed Nambooripad's partial order \cite{Nambooripad} on a regular semigroup. In 2001, Khan and Lawson found another application: they used variants as a means for introducing an alternative to the group of units in some classes of non-monoidal regular semigroups \cite{Khan}. This was followed by a number of articles on the topic of variants and sandwich semigroups \cite{SandVar2,SandVar4,SandVar6,SandVar8,Tsyaputa1,Tsyaputa2,SandwichI,SandwichII,SandwichIII,Variants}, and a chapter in the monograph \cite{GaylMaz}. These results proved to be applicable in other fields as well, as sandwich operations naturally arise in representation theory \cite{Repr1,Repr2}, category theory \cite{CatThe}, topology \cite{Top1,Top2}, automata theory \cite{AutThe1,AutThe2}, classical groups \cite{ClGroups}, computational algebra \cite{CompAlg}, and more.

In this article, the theme of variants is combined with another theme - partitions and diagrams (visual representations of partitions). Diagram categories and algebras are ubiquitous in representation theory \cite{PartRepTh1,PartRepTh2}, statistical mechanics \cite{stMech1,stMech2,stMech3,stMech4}, knot theory \cite{knTh1,knTh2,knTh3,knTh4,knTh5} and more. Due to their importance and wide range of applications, they have induced significant interest, which has led to inception of important ideas, such as approaching diagram algebras via diagram monoids and twisted semigroup algebras \cite{TwAlg1,TwAlg2,TwAlg3,TwAlg4,TwAlg5,TwAlg6,PartRepTh1}. This was beneficial for the theory of semigroups as well, as diagram monoids have interesting structural and combinatorial properties. In terms of structure, they are closely related to certain transformation semigroups and are natural examples of regular $*$-semigroups. On the other hand, their combinatorial structure opened new directions for research in the area of combinatorial semigroup theory \cite{TwAlg2,Comb1,Comb2}.
In particular, partial Brauer algebras and monoids have recently received much attention \cite{TwAlg2,PartBr1,Comb2,PartBr2,PartBr4,PartBr5,PartBr6,PartBr7,PartBr9,PartBr10} and alongside their planar counterparts, Motzkin algebras and monoids, are an emerging topic in current research efforts.

Classification of variants of a semigroup is a natural goal, one that presents itself because of the very definition of a variant. In 2003 and 2004 respectively, Tsyaputa classified variants of the full transformation monoid over a finite set $\Tt_{n}$ \cite{Tsyaputa1}, and its counterpart containing partial maps, $\Pt_{n}$ \cite{Tsyaputa2}. In 2018, Dolinka and East classified sandwich semigroups of linear transformations \cite{DolinkaEastMat}, thereby covering variants of $\mathcal{M}_{n}(\mathbb{F})$ (the semigroup of all $n\times n$ matrices over a field $\mathbb{F}$), as well. Finally, in \cite{SandwichIII}, the authors classified sandwich semigroups of Brauer diagrams, which includes variants of the partial Brauer monoid $\Bd_{n}$, as well. In this article, we continue the theme by considering variants of the partial Brauer monoid $\PBd_{n}$.
The article is organised as follows. In Section \ref{s:Prelimi}, we introduce the notions and notation needed for understanding the rest of the article. In Section \ref{s:PartialBrauerMonoids}, we present the notions specific to partitions and diagrams, and we develop a toolbox for dealing with partial Brauer partitions. Section \ref{s:Variants} contains the crux of the article, where we present the combinatorial analysis from which we infer the recurrence relation describing the number of $\Ll$-classes ($\Rr$-classes) in a regular $\Dd$-class of a variant $\PBd^{\alpha}_{n}$. Furthermore, in Lemma \ref{l:muInEq}, we prove an inequality which is vital for proving the main result, Theorem \ref{t:IsomClass>0}, where we classify the variants of the form $\PBd^{\alpha}_{n}$, where the sandwich element $\alpha$ has non-zero rank. The remaining case, when the rank of the sandwich element is zero, is considered in Section \ref{s:rank_zero}.


\section{Preliminaries}
\label{s:Prelimi}

Let $S$ be a semigroup. Recall that Green's preorders on $S$ are defined, for $x,y\in S$ by
\begin{equation*}
    \begin{gathered}
        x\leq_{\Rr} y \Leftrightarrow  xS^{1}\subseteq yS^{1}, \hspace{1cm}
        x\leq_{\Ll} y \Leftrightarrow  S^{1}x\subseteq S^{1}y, \hspace{1cm}
        x\leq_{\Jj} y \Leftrightarrow  S^{1}xS^{1}\subseteq S^{1}yS^{1},\\
    \end{gathered}
\end{equation*}
where $S^{1}$ is the monoid obtained from $S$ by adjoining an identity element 1, if necessary. Then, Green's relations of $S$ are defined as follows: for $\Kk\in\{\Rr,\Ll,\Jj\}$, we define $\Kk=\leq_{\Kk}\cap\geq_{\Kk}$, and we combine these to obtain $\Hh=\Rr\cap\Ll$ and $\Dd=\Rr\circ\Ll=\Ll\circ\Rr$. These relations are clearly equivalences. For $x\in S$, and for $\Kk\in\{\Rr,\Ll,\Hh,\Dd,\Jj\}$, let $\Kkk_{x}$ denote the $\Kk$-class of $S$ containing $x$. 

For $T\subseteq S$, let $\Ee(T)=\{x\in T:x=x^{2}\}$ denote the set of all idempotents of $S$ that belong to $T$. An element $x\in S$ is \emph{regular} if $x=xyx$ and $y=yxy$ for some $y\in S$, and $\Reg(S)$ denotes the set of all regular elements in $S$. It is well-known that for $x\in\Reg(S)$ we have $\Ddd_{x}\subseteq\Reg(S)$.

It is easily seen that any semigroup homomorphism preserves Green's classes and maps idempotents to idempotents, as well as regular elements to regular elements.
If $S$ is a monoid with identity $1$, let
$G(S)=\{x\in S: (\exists y\in S)\ xy=yx=1\}$ denote the group of units of $S$. 

Next, let $S$ be a semigroup, and consider the variant $S^{a}=(S,\star_{a})$, where
$x\star_{a}y=xay$ for $x,y\in S$. To avoid confusion, we denote the Green's relations of $S^{a}$ by $\Rr^{a}$, $\Ll^{a}$, $\Hh^{a}$, $\Dd^{a}$ and $\Jj^{a}$. Similarly, for each $\Kk\in\{\Rr,\Ll,\Hh,\Dd,\Jj\}$ and each $x\in S$, $\Kkk^{a}_{x}$ denotes the $\Kk^{a}$-class of $S^{a}$ containing $x$. As in \cite{SandwichI}, we define the \emph{P-sets} of $S^{a}$:
\begin{equation*}
    \begin{aligned}
        P^{a}_{1}=&\{x\in S:xa \Rr x\} \hspace{2cm} P^{a}_{3}=\{x\in S:axa\Jj x\}\\
        P^{a}_{2}=&\{x\in S:ax \Ll x\} \hspace{2cm} P^{a}=P^{a}_{1} \cap P^{a}_{2}
    \end{aligned}
\end{equation*}
These sets shape the Green's classes of $S^{a}$, as proved in \cite[Theorem 2.13]{DolinkaEastMat} (the partial semigroup in the proof is simply the semigroup $S$). We state the result for convenience.
\begin{thm}
\label{t:GreenSandSgp}
    Let $S$ be a semigroup with $a\in S$. In the variant $S^{a}$, we have
    \setlength{\columnsep}{-0.7cm}
    \begin{multicols}{2}
    \begin{enumerate}[(i)]
        \item  $\Rrr^{a}_{x}=\left\{
                             \begin{array}{ll}
                                \Rrr_{x}\cap \Pp^{a}_{1}, & \hbox{\ if $x\in     \Pp^{a}_{1}$} \\
                                \{x\}, & \hbox{\ if $x\in S\setminus \Pp^{a}_{1}$,}
                             \end{array}
                             \right.$
        \item  $\Lll^{a}_{x}=\left\{
                             \begin{array}{ll}
                                \Lll_{x}\cap \Pp^{a}_{2}, & \hbox{\ if $x\in \Pp^{a}_{2}$} \\
                                \{x\}, & \hbox{\ if $x\in S\setminus \Pp^{a}_{2}$,}
                             \end{array}
                             \right.$
        \item  $\Hhh^{a}_{x}=\left\{
                            \begin{array}{ll}
                               \Hhh_{x}, & \hbox{\ \ \ \ \ if $x\in \Pp^{a}$} \\
                               \{x\}, & \hbox{\ \ \ \ \ if $x\in S\setminus \Pp^{a}$,}
                            \end{array}
                            \right.$
        \item  $\Ddd_{x}^{a}=\left\{
                             \begin{array}{ll}
                                \Ddd_{x}\cap \Pp^{a}, & \hbox{if $x\in \Pp^{a}$} \\
                                \Lll^{a}_{x}, & \hbox{if $x\in \Pp^{a}_{2}\setminus \Pp^{a}_{1}$} \\
                                \Rrr^{a}_{x}, & \hbox{if $x\in \Pp^{a}_{1}\setminus \Pp^{a}_{2}$} \\
                                \{x\}, & \hbox{if $x\in S\setminus (\Pp^{a}_{1}\cup \Pp^{a}_{2})$,}
                            \end{array}
                            \right.$
        \item  $\Jjj^{a}_{x}=\left\{
                             \begin{array}{ll}
                                \Jjj_{x}\cap \Pp^{a}_{3}, & \hbox{\ \ if $x\in \Pp^{a}_{3}$} \\
                                \Ddd^{a}_{x}, & \hbox{\ \ if $x\in S\setminus \Pp^{a}_{3}$.}
                            \end{array}
                            \right.$                
    \end{enumerate}
    \end{multicols}
    If $x\in S\setminus \Pp^{a}$, then $\Hhh^{a}_{x}=\{x\}$ is a non-group $\Hh^{a}$-class in $S^{a}$.
\end{thm}

\section{Partial Brauer monoids}
\label{s:PartialBrauerMonoids}

Let $\mathbb{N}=\{0,1,2,\ldots\}$ denote the set of all natural numbers. For an integer $n\geq 1$ write $[n]=\{1,\ldots,n\}$, and write $[0]=\emptyset$. For $A\subseteq \mathbb{N}$, let $A'=\{a':a\in A\}$. Now, for $n\in\mathbb{N}$, let 
$\PBd_{n}$ denote the set of all partitions of the set $[n]\cup [n]'$ into blocks of size at most 2. A partition $\alpha\in\PBd_{n}$ may be visually presented in the form of a diagram, consisting of two rows of $n$ vertices corresponding the elements of $[n]$ and $[n]'$ (increasing from left to right), where the vertex $i$ is directly above the vertex $i'$, and the elements of the same block are connected by a line drawn inside the rectangle formed by these vertices. In Figure \ref{f:exampleDiagMult}, we present such diagrams for the partitions:
\begin{equation*}
    \begin{aligned}
        \alpha&=\{ \{1,5\},\{2\},\{3,2'\},\{4\},\{6,5'\},\{7,7'\},\{1',6'\},\{3',4'\} \}\in\PBd_{7},\\
        \beta&=\{ \{1,2\},\{3,2'\},\{4\},\{5,7\},\{6,6'\},\{1',3'\},\{4',7'\},\{5'\} \}\in\PBd_{7}.
    \end{aligned}
\end{equation*}
It is easily seen that the diagram representing a partial Brauer partition is unique in terms of vertices and edges it contains (which is not necessarily true for partitions of other types). Note that we identify the elements of the symmetric group $\mathcal{S}_{n}$ with the corresponding elements of the group of units $G(\PBd_{n})$.

Blocks containing elements of both sets ($[n]$ and $[n]'$) are called \emph{transversals}. All other blocks are \emph{non-transversals}. Non-singleton, non-transversal blocks are called \emph{hooks} (upper or lower, if their elements belong to $[n]$ or $[n]'$, respectively). The number of transversals in $\alpha$ is the \emph{rank} of $\alpha$, denoted $\rank(\alpha)$. Further, for $\alpha\in\PBd_{n}$, we define the \emph{domain}, \emph{codomain}, \emph{kernel}, \emph{cokernel} and the sets of all nontransversal upper and lower blocks:
\begin{equation*}
    \begin{aligned}
        \dom(\alpha) &= \{x\in [n]:x \text{ belongs to a transversal of }\alpha\},\\
        \codom(\alpha) &= \{x\in [n]:x' \text{ belongs to a transversal of }\alpha\},\\
        \ker(\alpha) &= \{(x,y)\in [n]\times[n]: x \text{ and } y \text{ belong to the same block of } \alpha\},\\
        \coker(\alpha) &= \{(x,y)\in [n]\times[n]: x' \text{ and } y' \text{ belong to the same block of } \alpha\},\\
        \Nu(\alpha) &= \{X\in\alpha: X \text{ is an upper non-transversal block of }\alpha\},\\
        \Nl(\alpha) &= \{X\in\alpha: X' \text{ is a lower non-transversal block of }\alpha\}.
    \end{aligned}
\end{equation*}
For example, in the partition $\alpha$ from Figure \ref{f:exampleDiagMult}, we have $\rank(\alpha)=3$, $\dom(\alpha)=\{3,6,7\}$, $\codom(\alpha)=\{2,5,7\}$, and the non-trivial kernel class (the only upper hook) is $\{1,5\}$, while the non-trivial cokernel classes (i.e.\ the lower hooks) are $\{1,6\}$ and $\{3,4\}$.

For $n\in\mathbb{N}$ and partitions $\alpha,\beta\in\PBd_{n}$, we define the product diagram $\Pi(\alpha,\beta)$ in the following way: 
\begin{itemize}
    \item we modify the diagram representing $\alpha$ by renaming each lower vertex $x'\in[n]'$ to $x''$, hence obtaining the graph $\alpha_{\downarrow}$ on $[n]\cup [n]''$;
    \item we modify the diagram representing $\beta$ by renaming each upper vertex $x\in[n]$ to $x''$, hence obtaining the graph $\beta^{\uparrow}$ on $[n]''\cup [n]$;
    \item we identify the vertices of the set $[n]''$ in $\alpha_{\downarrow}$ with the corresponding vertices of $[n]''$ in $\beta^{\uparrow}$, and obtain the graph $\Pi(\alpha,\beta)$.
\end{itemize}
Finally, the product partition $\alpha\beta$ of $\alpha$ and $\beta$ is the partial Brauer partition on $[n]\cup[n]'$ defined in the following way: for distinct $i,j\in[n]\cup[n]'$ we have $\{i,j\}\in\alpha\beta$ if, and only if, vertices $i$ and $j$ in $\Pi(\alpha,\beta)$ are connected by a path. In Figure \ref{f:exampleDiagMult}, we provide an example illustrating this calculation.
\begin{figure}
    \begin{center}
        \includegraphics[width=0.7\textwidth]{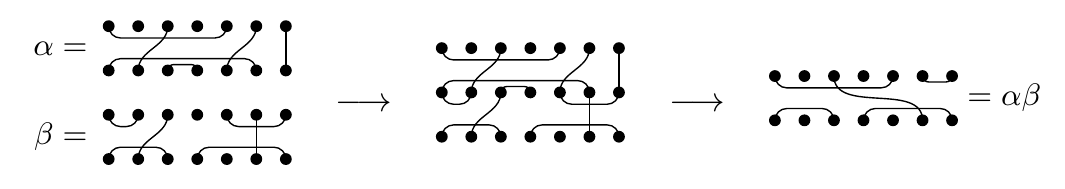}
        \caption{Multiplication of partitions $\alpha$ and $\beta$ via the product diagram $\Pi(\alpha,\beta)$.}
        \label{f:exampleDiagMult}
    \end{center}
\end{figure}

In addition to these standard notions, we will also need some novel ones, in order to present our results effectively. As in \cite{SandwichIII}, we say an equivalence $\varepsilon$ is a \emph{1-2-equivalence} if each $\varepsilon$-class has size at most $2$.
We introduce a new term, tailored to partial Brauer partitions. Let $\varepsilon$ be an equivalence on a set $T$, and let $X\subseteq T$. The pair $(\varepsilon,X)$ is a \emph{PB-pair} on $T$, if $\varepsilon$ is a 1-2-equivalence, and each element of $X$ belongs to a singleton $\varepsilon$-class. Note that any PB-pair on $[n]$ is a kernel-domain pair $(\ker(\alpha),\dom(\alpha))$, for some $\alpha\in\PBd_{n}$ (and similarly, there exists $\beta\in\PBd_{n}$ such that $(\coker(\beta),\codom(\beta))$ is the targeted pair). Thus, the elements of $X$ are called the \emph{domain} elements of the PB-pair, and $|X|$ is the \emph{rank} of the PB-pair. 

Let $(\varepsilon_{1},X_{1})$ and $(\varepsilon_{2},X_{2})$ be PB-pairs on a set $T$ and consider $(\varepsilon_{1}\vee\varepsilon_{2}, Z)$, where 
$$Z=\{(x,y)\in X_{1}\times X_{2}:x \text{ and } y \text{ belong to the same class of } \varepsilon_{1}\vee\varepsilon_{2}\}.$$
We say that $(\varepsilon_{1}\vee\varepsilon_{2}, Z)$ is the \emph{join} of the PB-pairs $(\varepsilon_{1},X_{1})$ and $(\varepsilon_{2},X_{2})$, and we denote it $(\varepsilon_{1},X_{1})\vee(\varepsilon_{2},X_{2})$.
Note that, in general, the join of two PB-pairs is not a PB-pair. The set $Z$ is called the \emph{domain} of the join, and $|Z|$ is the \emph{rank} of the join, denoted $\rank((\varepsilon_{1},X_{1})\vee(\varepsilon_{2},X_{2}))$. Note that the rank of the join does not depend on the order of PB-pairs (in other words, the join $(\varepsilon_{2},X_{2})\vee(\varepsilon_{1},X_{1})$ has the same rank). Furthermore, each element of $X_{1}$ ($X_{2}$) occurs in at most one pair of $Z$, as $\varepsilon_{1}$ and $\varepsilon_{2}$ are 1-2-equivalences, and elements of the domains belong to singleton classes of the corresponding equivalence. Just as a PB-pair represents a half of some partition, the join of PB-pairs represents identifying vertices in some product diagram (where $\varepsilon_{1}\vee\varepsilon_{2}$ is the resulting equivalence on the middle row and $Z$ contains the terminal vertices of the paths that will determine the transversals of the product). 

To visually present a PB-pair $(\varepsilon, X)$ on a finite set $T$, we will use the same technique as for diagrams. We arrange $|T|$ vertices in a row, and identify them with the elements of $T$ (if $T=[n]$ for some $n\in\mathbb{N}$, we arrange them in the ascending order). Then, elements belonging to the same $\varepsilon$-class are connected by a line drawn above the vertices. Finally, each of the elements of the domain is the starting point of an upward straight line. Effectively, it means drawing a half-diagram, as in Figure \ref{f:sigmacoker}.
\begin{figure}
    \begin{center}
        \includegraphics[width=0.7\textwidth]{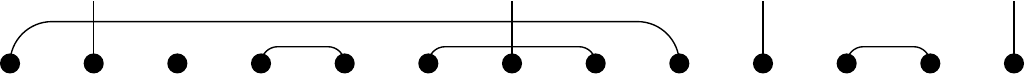}
        \caption{A visual presentation of a PB-pair.} 
        \label{f:sigmacoker}
    \end{center}
\end{figure}
We may also visually present the join of two PB-pairs. This will correspond to the middle part of a product diagram. Namely, in the join of the PB-pairs $(\varepsilon_{1},X_{1})$ and $(\varepsilon_{2},X_{2})$ on a set $T$, the first PB-pair is drawn in the way described above and the second one is drawn on the same set of vertices, but all the lines corresponding to the $\varepsilon_{2}$-connections and elements of $X_{2}$ will be drawn below the vertices, as in Figure \ref{f:joinOfPB-pairs}.
\begin{figure}
    \begin{center}
        \includegraphics[width=0.7\textwidth]{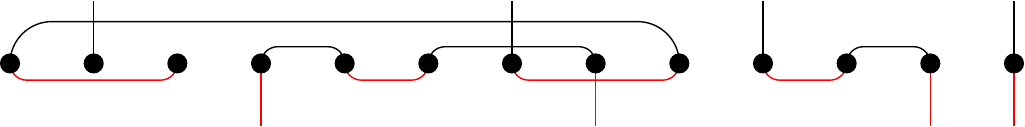}
        \caption{A visual presentation of the join of two PB-pairs.} 
        \label{f:joinOfPB-pairs}
    \end{center}
\end{figure}
Clearly, the elements of the domain of the join correspond to the \emph{domain paths} - paths connecting elements of $X_{1}$ and $X_{2}$ (including the trivial paths, as well). Such a path connects an upward straight line and a downward straight line, while the rest of the path is made up by hooks. In fact, such a component of the graph is of form
\begin{description}
   \item[(P)\label{(P)}] $u\xlongleftrightarrow{\varepsilon_{2}} w_{1}\xlongleftrightarrow{\varepsilon_{1}} \cdots \xlongleftrightarrow{\varepsilon_{2}} w_{2k-1} \xlongleftrightarrow{\varepsilon_{1}} v$, for some $u\in X_{1}$, $v\in X_{2}$, $k\geq 1$ and $w_{1},\ldots,w_{2k-1}\in[n],$\\[-10pt]
\end{description}
(If $u=v$, we have a trivial domain path.) Note that each of the vertices $w_{1},\ldots,w_{2k-1}$ belongs to a non-singleton $\varepsilon_{1}$-class and to a non-singleton $\varepsilon_{2}$-class.

This graphical construction is quite similar to graphs $\Lambda(\alpha)$ and $\Gamma_{\alpha}$ in \cite{TwAlg2} and \cite{TwAlg7}, respectively. In both cases the graph was to determine whether $\alpha$ is an idempotent, which boils down to whether $\rank(\alpha^{2})=\rank(\alpha)$. Here, our graphs will be used to determine if $\rank(\alpha\xi)=\rank(\alpha)$, which is a more general problem.

In our discussion, we will rely on the notions and conclusions made here, even without explicitly mentioning the visual presentations.

\section{Variants of $\PBd_{n}$}
\label{s:Variants}

Let $n\in\mathbb{N}$ and let $\alpha\in\PBd_{n}$. Put $r=\rank \alpha$ and consider the variant $\PBd^{\alpha}_{n}$. In \cite[Proposition 5.7]{SandwichIII} it was proved that
\begin{equation}
\label{e:PsetsPB}
    \begin{aligned}
        \Pp^{\alpha}_{1}&=\{\xi:\PBd_{n}:\rank(\xi\alpha)=\rank(\xi)\},\\
        \Pp^{\alpha}_{2}&=\{\xi:\PBd_{n}:\rank(\alpha\xi)=\rank(\xi)\},\\
        \Reg(\PBd^{\alpha}_{n})=\Pp^{\alpha}=\Pp^{\alpha}_{3}&=
        \{\xi\in\PBd_{n}:\rank(\xi\alpha)=\rank(\alpha\xi)=\rank(\xi)\}\\
        &=\{\xi\in\PBd_{n}:\rank(\alpha\xi\alpha)=\rank(\xi)\}.
    \end{aligned}
\end{equation}
Furthermore, we have
\begin{prop}
\label{p:regClassesPB}
    If $\xi\in\Pp^{\alpha}=\Reg(\PBd^{\alpha}_{n})$, then
    \begin{enumerate}[(i)]
        \item $\Rrr^{\alpha}_{\xi}=\Rrr_{\xi}\cap\Pp^{\alpha}=\{\sigma\in\Pp^{\alpha}:\ker(\sigma)=\ker(\xi), \text{ and } \dom(\sigma)=\dom(\xi)\}$,
        \item $\Lll^{\alpha}_{\xi}=\Lll_{\xi}\cap\Pp^{\alpha}=\{\sigma\in\Pp^{\alpha}:\coker(\sigma)=\coker(\xi), \text{ and } \codom(\sigma)=\codom(\xi)\}$,
        \item $\Ddd^{\alpha}_{\xi}=\Ddd_{\xi}\cap\Pp^{\alpha}=\{\sigma\in\Pp^{\alpha}:\rank(\sigma)=\rank(\xi)\}$.
    \end{enumerate}
    Thus, the regular $\Jj^{\alpha}=\Dd^{\alpha}$-classes of $\PBd^{\alpha}_{n}$ are precisely the sets
    \begin{equation*}
        \Ddd^{\alpha}_{k}=\Ddd_{k}\cap\Pp^{\alpha}=\{\xi\in\Pp^{\alpha}:
        \rank(\xi)=k\}\hspace{1cm} \text{ for each } 0\leq k\leq r.
    \end{equation*}
    These form a chain under the usual ordering $\Jj^{\alpha}$-classes: $\Ddd^{\alpha}_{k}\leq\Ddd^{\alpha}_{l}\Leftrightarrow k\leq l$.
\end{prop}
\begin{proof}
    Part $(iii)$ follows from \cite[Proposition 5.8]{SandwichIII}. We prove only $(i)$, as the proof for $(ii)$ is dual. Let $\xi\in\Pp^{\alpha}$. Since $\Rrr^{\alpha}_{\xi}\subseteq\Ddd^{\alpha}_{\xi}\subseteq\Reg(\PBd^{\alpha}_{n})=\Pp^{\alpha}$, Theorem \ref{t:GreenSandSgp}(i) gives
    $$\Rrr^{\alpha}_{\xi}=\Rrr^{\alpha}_{\xi}\cap\Pp^{\alpha}=\Rrr_{\xi}\cap\Pp^{\alpha}_{1}\cap\Pp^{\alpha}=\Rrr_{\xi}\cap\Pp^{\alpha}.$$
    The last equality in $(i)$ follows from \cite[Theorem 4.9(iv)]{SandwichIII}. The statement about the regular $\Dd^{\alpha}$-classes was proved in \cite[Proposition 5.8(i)]{SandwichIII}.
\end{proof}

Now, we consider a regular $\Dd^{\alpha}$-class, with the aim to calculate the number of $\Ll^{\alpha}$-classes in it. Let $0\leq q\leq r=\rank(\alpha)$, and consider an $\Ll^{\alpha}$-class in $\Ddd^{\alpha}_{q}$. By Lemma \ref{p:regClassesPB}(ii), such a class is uniquely determined by the properties of the lower row of its elements (the cokernel-codomain combination). If we want to enumerate these classes, we need to know which cokernel-codomain pairs occur in elements of $\Ddd^{\alpha}_{q}$. Such a pair is a PB-pair on $[n]=\{1,\ldots,n\}$. 

In order to calculate the number of these PB-pairs, we will introduce additional notation. For $m\in\mathbb{N}$ and $0\leq k\leq m$ with $k\equiv m\ (\mod1 2)$, let us fix the equivalence $\varepsilon_{m,k}$ with classes $\{1\},\ldots,\{k\},\{k+1,k+2\},\ldots,\{m-1,m\}$. Furthermore, for for $m,k,t,q\in\mathbb{N}$, let $\mu(m,k,t,q)$ be the number of PB-pairs $(\eta,X)$ such that $|X|=q$ and $\rank((\varepsilon_{m,k},[t])\vee(\eta,X))=q$. If the equivalence $\varepsilon_{m,k}$ is undefined, or $(\varepsilon_{m,k},[t])$ is not a PB-pair, we fix $\mu(m,k,t,q)=0$. Note that, by symmetry, $\mu(m,k,t,q)$ is also the number of PB-pairs $(\eta,X)$ such that $|X|=q$ and the rank of $(\eta,X)\vee(\varepsilon_{m,k},[t])$ is $q$.
Now, we may prove the following two lemmas.

\begin{lem}\label{l:ddd/ll}
    Let $n\in\mathbb{N}$ and $\alpha\in\PBd_{n}$. Put $r=\rank \alpha$ and let $k$ denote the number of singleton classes in $\ker(\alpha)$. For $0\leq q\leq r$, in the variant $\PBd^{\alpha}_{n}$, we have $|\Ddd^{\alpha}_{q}/\Ll^{\alpha}|=\mu(n,k,r,q)$.
\end{lem}
\begin{proof}
    Let $0\leq q\leq r=\rank(\alpha)$. As we noted above, $|\Ddd^{\alpha}_{q}/\Ll^{\alpha}|$ is the number of all PB-pairs on $[n]$ that occur as cokernel-codomain pairs in the elements of $\Ddd^{\alpha}_{q}$. Let $(\eta,X)$ be such a PB-pair. Firstly, note that $|X|=q$, since $\Ddd^{\alpha}_{q}$ contains only elements of rank $q$. Secondly, note that the join $(\eta,X)\vee(\ker(\alpha),\dom(\alpha))$ has rank $q$. Let us elaborate. Since $\eta=\coker(\xi)$ and $X=\codom(\xi)$ for some $\xi\in\Pp^{\alpha}\subseteq\Pp^{\alpha}_{1}$, from (\ref{e:PsetsPB}) follows $\rank(\xi\alpha)=\rank(\xi)$, which means that the join $(\eta,X)\vee(\ker(\alpha),\dom(\alpha))$ has rank $|X|$.
    
    Thirdly, we show that each PB-pair $(\eta, X)$ on $[n]$ such that
    \begin{enumerate}[(a)]
        \item\label{(a1)} $|X|=q$, and
        \item\label{(b1)} the join $(\eta,X)\vee(\ker(\alpha),\dom(\alpha))$ has rank $q$,
    \end{enumerate}
    occurs as a cokernel-codomain pair of an element from $\Ddd^{\alpha}_{q}$. Let $(\eta,X)$ be a PB-pair on $[n]$ that satisfies these requirements. Let $\beta\in\Ddd^{\alpha}_{q}$, and consider a partition $\gamma\in\PBd_{n}$ with %
    $$\ker(\gamma)=\ker(\beta),\ \dom(\gamma)=\dom(\beta),\ \ \text{and}\ \  \coker(\gamma)=\eta,\ \codom(\gamma)=X.$$
    (such a partial Brauer partition clearly exists). Then, the first two properties and Proposition (\ref{e:PsetsPB}) imply $\gamma\in\Pp^{\alpha}_{2}$ (because $\rank(\alpha\beta)=\rank(\beta)$, so $\rank(\alpha\gamma)=\rank(\gamma)$). Similarly, the last two properties and (\ref{e:PsetsPB}) imply $\gamma\in\Pp^{\alpha}_{1}$ (because the join $(\eta,X)\vee(\ker(\alpha),\dom(\alpha))$ has rank $q$). Thus, we have $\gamma\in\Pp^{\alpha}$. Since $\rank(\gamma)=|X|=q$, from Proposition \ref{p:regClassesPB} we deduce $\gamma\in\Ddd^{\alpha}_{q}$.
    
    We have proved that $|\Ddd^{\alpha}_{q}/\Ll^{\alpha}|$ is the number of all PB-pairs on $[n]$ satisfying (\ref{(a1)}) and (\ref{(b1)}). We claim that the number of such pairs is $\mu(m,k,r,q)$. Recall that $k$ denotes the number of singletons in the partition corresponding $\ker(\alpha)$. Note that the PB-pairs $(\ker(\alpha),\dom(\alpha))$ and $(\varepsilon_{n,k},[r])$ have equal ranks and their equivalences have the same number of singletons. Thus, there exists a bijection $\phi\in\Ss_{n}$ mapping $\dom(\alpha)$ to $[r]$, and the classes of $\ker(\alpha)$ to the classes of $\varepsilon_{n,k}$. Hence, for a PB-pair $(\eta,X)$ on $[n]$ satisfying (\ref{(a1)}) and (\ref{(b1)}), we may define a PB-pair $(\eta_{\phi},X_{\phi})$ on $[n]$ with $\eta_{\phi}=\{(x\phi,y\phi):(x,y)\in\eta\}$ and $X_{\phi}=\{x\phi:x\in X\}$. Since $\phi$ maps $\dom(\alpha)$ to $[r]$, and since it maps the classes of $\ker(\alpha)$ to the classes of $\varepsilon_{n,k}$, the rank of the join $(\eta_{\phi},X_{\phi})\vee(\varepsilon_{n,k},[r])$ is $q$. It easily seen that  $(\eta,X)\mapsto (\eta_{\phi},X_{\phi})$ is a bijection mapping PB-pairs on $[n]$ satisfying (\ref{(a1)}) and (\ref{(b1)}) to PB-pairs on $[n]$ of rank $q$ such that the rank of their join with $(\varepsilon_{m,k},[r])$ is $q$. Therefore, $|\Ddd^{\alpha}_{q}/\Ll^{\alpha}|=\mu(m,k,r,q)$.
\end{proof}

Of course, the dual statement immediately follows: 
\begin{lem}
\label{l:ddd/rr}
    Let $n\in\mathbb{N}$ and $\alpha\in\PBd_{n}$. Put $r=\rank \alpha$ and let $k$ denote the number of singleton classes in $\coker(\alpha)$. For $0\leq q\leq r$, in $\PBd^{\alpha}_{n}$ holds:
    $$|\Ddd^{\alpha}_{q}/\Rr^{\alpha}|=\mu(n,k,r,q).$$
\end{lem}

In the following lemma, we will prove a recurrence describing the numbers $\mu(m,k,r,q)$.
\begin{lem}
\label{l:muRecurrence}
    For $n,k,r,q\in\mathbb{N}$, the numbers $\mu(n,k,r,q)$ satisfy the following recurrence: 
        \begin{enumerate}[(i)]
            \item $\begin{aligned}[t]\mu(n,k,r,q)=&(n-k)\mu(n-2,k,r,q)+\mu(n-1,k-1,r-1,q-1)+\mu(n-1,k-1,r-1,q)+\\
            &(k-r)\mu(n-2,k-2,r-1,q)+(r-1)\mu(n-2,k-2,r-2,q)
            \end{aligned}$\\
            if $n\geq k\geq r\geq q>0$ and $n\equiv k(\mod1 2)$.
            \item $\mu(n,k,r,0)=\sum^{\lfloor \frac{n}{2}\rfloor}_{i=0}\binom{n}{2i}(2i-1)!!$ if $n\geq k\geq r$ and $n\equiv k(\mod1 2)$.
            \item $\mu(n,k,r,q)=0$, otherwise.
\end{enumerate}
\end{lem}
\begin{proof}
    Recall that, for $n,k,r,q\in\mathbb{N}$, $\mu(n,k,r,q)$ is the number of PB-pairs $(\eta,X)$ such that $|X|=q$ and the rank of the join $(\varepsilon_{n,k},[r])\vee(\eta,X)$ is $q$. Consider the PB-pair $(\varepsilon_{n,k},[r])$. We may present it visually as in the example in Figure \ref{f:sigmacoker}.
    \begin{itemize}
            \item If $r\geq q>0$, consider the element $r$. The question is: what role can be played by this element in $\eta$? We have five cases:
                \begin{description}
                    \item[Case 1:] The element $r$ is an element of $X$ (the domain of our pair). Then, we have chosen one of the $q$ elements of $X$, so the remaining $n-1$ elements can be connected to construct a suitable PB-pair in $\lambda(n-1,k-1,r-1,q-1)$ ways.
                    \item[Case 2:] The element $r$ is a member of a singleton $\eta$-class, outside of $X$. Then, the remaining $n-1$ elements can be connected to construct a suitable PB-pair in $\lambda(n-1,k-1,r-1,q)$ ways.
                    \item[Case 3:] The element $r$ is connected to an element $b$ belonging to a two-element $\varepsilon_{n,k}$-class (i.e.\ $b\in\{k+1,\ldots,n\}$). There are $n-k$ such elements. In that case, these three elements ($r$, $b$, and $b$'s pair) may be considered as a single domain element. This element and the remaining $n-3$ elements can be connected to form a suitable PB-pair in $\lambda(n-2,k,r,q)$ ways.
                    \item[Case 4:] The element $r$ is connected to an element $b$, which forms a singleton $\varepsilon_{n,k}$-class and is outside of $[r]$ (i.e.\ $b\in\{r+1,\ldots,k\}$). There are $k-r$ such elements. Then, the remaining $n-2$ elements can be connected to construct a suitable PB-pair in $\lambda(n-2,k-2,r-1,q)$ ways.
                    \item[Case 5:] The element $r$ is connected to an element $b$, which belongs to $[r-1]$. Obviously, there are $r-1$ such elements). Then, the remaining $n-2$ elements can be connected to construct a suitable PB-pair in $\lambda(n-2,k-2,r-2,q)$ ways.
                \end{description}
                These cases are depicted in Figure \ref{f:fiveCases}. We add up the values in these five cases, and obtain the same recurrence as in the case $(i)$.
            \item If $q=0$, we fix $X=\emptyset$, and count all the possible 1-2-partitions of the set $[n]$. Note that the number of two-element classes can be any number between $0$ and $\lfloor \frac{n}{2}\rfloor$. If there are $i$ such classes, then the vertices belonging to these classes can be chosen in $\binom{n}{2i}$ ways (the remaining vertices form singleton $\varepsilon$-classes). These $2i$ elements can be paired in $(2i-1)!!$ ways. We obtain the same formula as in the case $(ii)$.
            \item If $r<q$, there exists no PB-pair that can generate a $q$-domain join with $(\varepsilon_{n,k},[r])$. Thus, in this case we have $\mu(n,k,r,q)=0$.
            \item If either $n\geq k\geq r$ or $n\equiv k(\mod1 2)$ is false, then these numbers do not correspond to any PB-pair, so in this case we have $\mu(n,k,r,q)=0$. This case and the previous one correspond to the part $(iii)$.
        \end{itemize}
    Thus, the numbers $\mu(n,k,r,q)$ satisfy the stated recurrence, and the result follows.
\end{proof}

\begin{figure}
\begin{center}
        \includegraphics[width=0.7\textwidth]{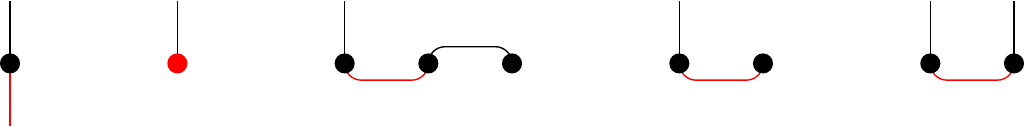}
        \caption{Cases 1--5 in the proof of Lemma \ref{l:muRecurrence}. The connections from $\varepsilon_{n,k}$ are coloured black, and those from $\eta$ are coloured red.} 
        \label{f:fiveCases}
\end{center}
\end{figure}

Now, we investigate the properties of the numbers $\mu(n,k,r,q)$.
\begin{lem}
\label{l:muInEq}
    Let $n,k,r,q\in\mathbb{N}$ with $n\geq k\geq r\geq q\geq 1$ and $n\equiv k(\mod1 2)$. If $n\geq k+2$, then
    $$\mu(n,k,r,q)>\mu(n,k+2,r,q).$$
\end{lem}
\begin{proof}
    Suppose $n,k,r,q\in\mathbb{N}$ satisfy all the assumptions of the lemma. Again, recall that $\mu(n,k,r,q)$ is the number of PB-pairs $(\eta,X)$ such that:
    \begin{itemize}
        \item[(I)] $|X|=q$ and $\rank((\varepsilon_{n,k},[r])\vee(\eta,X))=q$.
    \end{itemize}
    Similarly, $\mu(n,k+2,r,q)$ is the number of PB-pairs $(\eta,X)$ such that:
    \begin{itemize}
        \item[(II)] $|X|=q$ and $\rank((\varepsilon_{n,k+2},[r])\vee(\eta,X))=q$.
    \end{itemize}
    The PB pairs $(\varepsilon_{n,k},[r])$ and $(\varepsilon_{n,k+2},[r])$ are illustrated in Figure \ref{f:Epsilon_nk}, using an example.
    
\begin{figure}
\begin{center}
        \includegraphics[width=\textwidth]{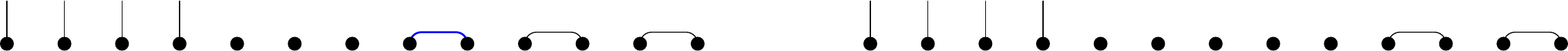}
        \caption{PB-pairs $(\varepsilon_{13,7},[4])$ (left) and $(\varepsilon_{13,9},[4])$ (right). Note that the only difference is the upper hook connecting $8$ and $9$ (coloured blue) in the first pair.} 
        \label{f:Epsilon_nk}
\end{center}
\end{figure}
    
    First, we will prove that all PB-pairs on $[n]$ satisfying (II) also satisfy (I). Recall from the discussion in Section \ref{s:PartialBrauerMonoids} that the domain elements of the join of two PB-pairs correspond to the paths connecting elements of their domains. Thus, in both cases we consider the domain paths connecting the elements of $[r]$ and $X$. Hence, these paths are of the form \ref{(P)}, so none of them contains elements that belong to non-domain singletons in either of the factors. Let $(\eta,X)$ be a PB-pair on $[n]$ satisfying (II). By the previous discussion, elements $k+1$ and $k+2$ do not belong to any of these domain paths in the join $(\varepsilon_{n,k+2},[r])\vee(\eta, X)$. So, all the $q$ domain paths exist also in the join $(\varepsilon_{n,k},[r])\vee(\eta, X)$. Furthermore, the connection of $k+1$ and $k+2$ in $\varepsilon_{n,k}$ does not generate a new domain path in the join $(\varepsilon_{n,k},[r])\vee(\eta, X)$, as we have $|X|=q$ and no element of $X$ can appear in two domain paths, so the join has at most $q$ domain paths. Thus, we have
    $$\mu(n,k,r,q)\geq\mu(n,k+2,r,q).$$
    
    We need to prove the strict inequality. Consider the PB-pair $(\eta,X)$, where $X=\{1,\ldots,q-1,k+2\}$ and $\eta$ is the equivalence on $[n]$ with a unique non-trivial class $\{q,k+1\}$. It is easily seen that $(\eta,X)$ is a PB-pair of rank $q$. In Figure \ref{f:lemma44}, we present the joins of $(\eta,X)$ with $\varepsilon_{n,k}$ and $\varepsilon_{n,k+2}$, respectively. 
    \begin{figure}
        \includegraphics[width=\textwidth]{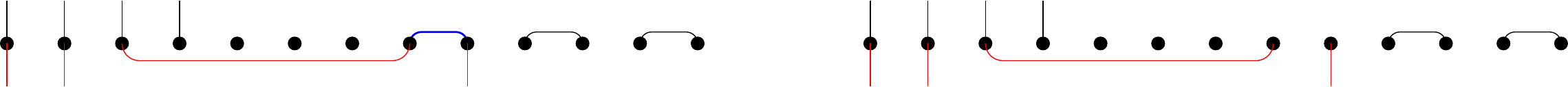}
        \caption{The joins of PB-pairs $(\varepsilon_{13,7},[4])$ (left) and $(\varepsilon_{13,9},[4])$ (right) with the pair $(\eta,X)$ in the case $q=3$. The connections belonging to the PB-pairs $(\varepsilon_{13,7},[4])$ and $(\varepsilon_{13,9},[4])$ are coloured black, and those that belong to $(\eta,X)$ are red.} 
        \label{f:lemma44}
    \end{figure}
    Note that the ranks of these joins are $q$ and $q-1$, respectively. Thus, the PB-pair $(\xi,Y)$ satisfies (I), but not (II).
    Therefore, we may conclude that
    $$\mu(n,k,r,q)>\mu(n,k+2,r,q).\qedhere$$
\end{proof}

Note that Lemma \ref{l:muInEq} applies only in the case when $r\geq 1$. Namely, it is false when $r=0$, since Lemma \ref{l:muRecurrence}(ii) implies that $\mu(n,k,r,0)$ depends only on $n$. Thus, the case $r=0$ will be discussed separately. In order to prove the classification result for $r\geq1$, we prove the following:
\begin{lem}
\label{l:Eq->isom}
    Let $n\in\mathbb{N}$, and let $\alpha,\beta\in\PBd_{n}$ with $\rank(\alpha)=\rank(\beta)$. In addition, write $k$ and $l$ for the number of singleton classes in $\ker(\alpha)$ and $\ker(\beta)$, respectively. Similarly, write $p$ and $w$ for the number of singleton classes in $\coker(\alpha)$ and $\coker(\beta)$, respectively.  If $k=l$ and $p=w$, then $\PBd^{\alpha}_{n}\cong\PBd^{\beta}_{n}$.
\end{lem}
\begin{proof}
    Suppose that $k=l$ and $p=w$. 
    From these equalities and $\rank(\alpha)=\rank(\beta)$, we have $\beta=\pi_{1}\alpha\pi_{2}$ for some permutations $\pi_{1},\pi_{2}\in\mathcal{S}_{n}=G(\PBd_{n})$. Then, $\alpha\mapsto \pi^{-1}_{2}\alpha\pi^{-1}_{1}$ determines an isomorphism $\PBd^{\alpha}_{n}\rightarrow\PBd^{\beta}_{n}$.
\end{proof}
 
Now, we want to prove the classification result. In order to to that, we will need the size of the underlying set of the variant $\PBd^{\alpha}_{n}$. As in \cite{SandwichIII}, we define the numbers $a(k)$ by $a(0)=a(1)=1$ and $a(k)=a(k-1)+(k-1)a(k-2)$ for $k\geq 2$. It is easy to see that $a(k)\geq 1$ for all $k\in\mathbb{N}$. Thus, $a(k)>a(k-1)$ for all $k\geq 2$.
In \cite[Proposition 4.4]{SandwichIII} it was noted that 
\begin{equation}
\label{e:cardPBm}
    |\PBd_{n}|=a(2n).
\end{equation}

\begin{thm}
\label{t:IsomClass>0}
    Let $m,n\in\mathbb{N}$, and let $\alpha\in\PBd_{m}$ and $\beta\in\PBd_{n}$ with $r=\rank(\alpha)\geq 1$ and $s=\rank(\beta)\geq 1$. In addition, write $k$ and $l$ for the number of singleton classes in $\ker(\alpha)$ and $\ker(\beta)$, respectively. Similarly, write $p$ and $w$ for the number of singleton classes in $\coker(\alpha)$ and $\coker(\beta)$, respectively.  Then $\PBd^{\alpha}_{m}\cong\PBd^{\beta}_{n}$ if and only if $m=n$, $k=l$, $p=w$, and $r=s$.
\end{thm}
\begin{proof}
    The reverse implication follows from Lemma \ref{l:Eq->isom}.
    
    Conversely, suppose that $\PBd^{\alpha}_{m}\cong\PBd^{\beta}_{n}$. 
    As $|\PBd_m|=a(2m)$ and $|\PBd_n|=a(2n)$, and $\mathbb{N}\rightarrow\mathbb{N}:x\mapsto a(2x)$ is an increasing function, we have $m=n$.
    Further, by Lemma \ref{p:regClassesPB}, the variants $\PBd^{\alpha}_{m}$ and $\PBd^{\beta}_{n}$ have $r+1$ regular $\Dd^{\alpha}$-classes and $s+1$ regular $\Dd^{\beta}$-classes, respectively. Since isomorphisms preserve regularity and Green's classes, we have $r=s$.
    By Lemma \ref{l:ddd/ll}, in variants $\PBd^{\alpha}_{m}$ and $\PBd^{\beta}_{n}$, for $q=r$ we have
    $$\mu(n,k,r,r)=|\Ddd^{\alpha}_{r}/\Ll^{\alpha}|\ \ \ \text{and}\ \ \  |\Ddd^{\beta}_{q}/\Ll^{\beta}|=\mu(n,l,r,r),$$
    respectively. Since $r\geq 1$ and isomorphism preserves the number of these $\Ll$-classes, Lemma \ref{l:muInEq} gives $k=l$.
    Similarly, Lemma \ref{l:ddd/rr} and Lemma \ref{l:muInEq} give $p=w$.
\end{proof}

\begin{rem}
    Note that, in the proof of Theorem \ref{t:IsomClass>0}, we have shown that $\PBd^{\alpha}_{m}\cong\PBd^{\beta}_{n}$ implies $m=n$ and $r=s$ (this part of the proof does not require the assumption $r,s\geq 1$).
\end{rem}

\section{Sandwich element of rank zero}
\label{s:rank_zero}

Our proof of the classification result does not apply in the case when the sandwich element has rank $0$. We consider this case separately. It is natural to expect that the same criterion holds:
\begin{hyp}
\label{h:rank0}
    Let $n\in\mathbb{N}$, and let $\alpha,\beta\in\PBd_{n}$ with $\rank(\alpha)=\rank(\beta)=0$. In addition, write $k$ and $l$ for the number of singleton classes in $\ker(\alpha)$ and $\ker(\beta)$, respectively. Similarly, write $p$ and $w$ for the number of singleton classes in $\coker(\alpha)$ and $\coker(\beta)$, respectively. Then $\PBd^{\alpha}_{n}\cong\PBd^{\beta}_{n}$ if and only if $k=l$ and $p=q$.
\end{hyp}

The reverse implication was proved in Lemma \ref{l:Eq->isom}. However, for the direct implication, our previous approach does not work, since we have $\mu(n,k,0,0)=\mu(n,l,0,0)$ even if $k\neq l$. Furthermore, since $\rank(\alpha)=\rank(\beta)=0$, by (\ref{e:PsetsPB}) we have
\begin{equation*}
    \Pp^{\alpha}_{1}=\Pp^{\alpha}_{2}=\Pp^{\alpha}=\{\sigma\in\PBd_{n}:\rank(\sigma)=0\}=\Pp^{\beta}=\Pp^{\beta}_{1}=\Pp^{\beta}_{2},   
\end{equation*}
which is the unique regular $\Dd^{\alpha}$- and $\Dd^{\beta}$-class in $\PBd^{\alpha}_{n}$ and $\PBd^{\beta}_{n}$, respectively. It is easily seen that this set is a rectangular band, where the sandwich operation coincides with the original operation (for any sandwich element).
Therefore, we cannot hope for any progress by considering the regular subsemigroup. In addition, all the elements of $\PBd_{n}\setminus\Pp^{\alpha}$ form singleton $\Jj^{\alpha}$-classes ($\Jj^{\beta}$-classes), which are above the regular $\Dd^{\alpha}=\Jj^{\alpha}$-class ($\Dd^{\beta}=\Jj^{\beta}$-class), and no other pair of $\Jj^{\alpha}$-classes ($\Jj^{\beta}$-classes) are related. For this reason, even if $k\neq l$, the $\leq_{\Jj}$-structure of the semigroups is the same.

In spite of these similarities, these variants are not necessarily isomorphic. To show that, we need to consider the Green's preorders $\leq_{\Ll^{\alpha}}$ and $\leq_{\Rr^{\alpha}}$ on $\PBd^{\alpha}_{n}$.
Since $\id_{n}$ is the left- and right- identity for $\PBd_{n}$, \cite[Remark 3.8]{SandwichIII} gives 
\begin{equation}
\label{e:pom1}
    \sigma\leq_{\Rr^{\alpha}}\tau\Leftrightarrow \sigma\leq_{\Rr}\tau\alpha \ \ \ \ \ \text{and}\ \ \ \ \ \sigma\leq_{\Rr^{\beta}}\tau\Leftrightarrow \sigma\leq_{\Rr}\tau\beta,
\end{equation}
for $\sigma\in\Pp^{\alpha}_{1}=\Pp^{\alpha}=\Pp^{\beta}=\Pp^{\beta}_{1}$ and $\tau\in\PBd_{n}$. From \cite[Theorem 8]{PartBr6} follows the characterisation of relation $\leq_{\Rr}$ in $\PBd_{n}$:
\begin{equation}
\label{e:pom2}
    \sigma\leq_{\Rr}\tau\Leftrightarrow \Nu(\sigma)\supseteq\Nu(\tau)    
\end{equation}
for $\sigma,\tau\in\PBd_{n}$.
Now, we provide a simple example in which the difference in the case $k\neq l$ is easily seen:
\begin{exm}\label{ex:kraj}
    Let $n=2$ and consider $\alpha=\disppartnx{2}{2}{
                                   \darcx12{.2}
    }$ and
    $\beta=\disppartnx{2}{2}{
          \uarcx12{.2}
          \darcx12{.2}
    }$. 
    Then, $k=2$, $l=0$ and $p=w=0$. Using (\ref{e:pom1}) and (\ref{e:pom2}), one can easily verify that Figure \ref{f:nonisomRank0} depicts the $\leq_{\Rr^{\alpha}}$-relations and $\leq_{\Rr^{\beta}}$-relations in $\PBd^{\alpha}_{2}$ and $\PBd^{\beta}_{2}$, respectively.
    Since isomorphisms preserve Green's preorders and Green's classes, we clearly have $\PBd^{\alpha}_{2}\not\cong\PBd^{\beta}_{2}$.
    \begin{figure}[h]
    \begin{center}
        \includegraphics[width=0.85\textwidth]{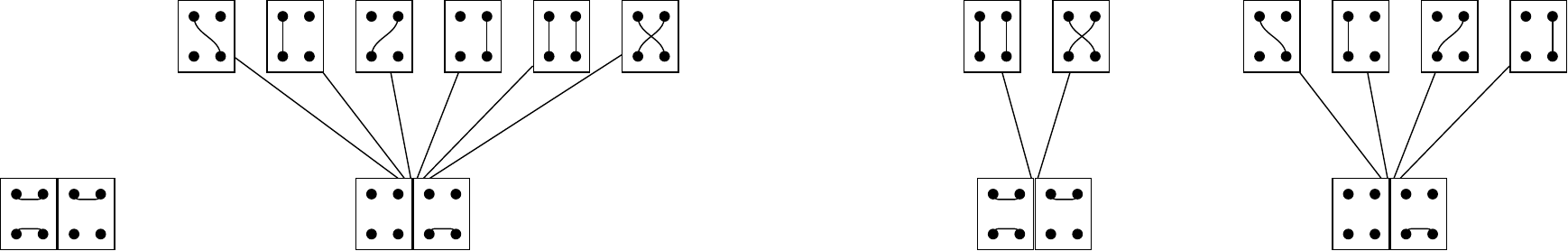}
        \caption{The relations among the $\Rr^{\alpha}$-classes of $\PBd^{\alpha}_{2}$ (left) and the relations among the $\Rr^{\beta}$-classes of $\PBd^{\beta}_{2}$ (right).}
        \label{f:nonisomRank0}
    \end{center}
    \end{figure}
\end{exm}

Thus, in order to prove Hypothesis \ref{h:rank0}, one should consider the distribution of non-regular $\Rr^{\alpha}$-classes above the regular ones and symmetrically, the distribution of non-regular $\Ll^{\alpha}$-classes above the regular ones. 

We give an additional result which might prove useful in this direction of investigation. Recall from \cite{Tamura} that a semigroup $S$ is an \emph{inflation} of a semigroup $T$ if $T$ is a subsemigroup of $S$ and there exists a mapping $\phi:S\rightarrow T$ such that $\phi(\sigma)=\sigma$ for $\sigma\in T$ and
\begin{equation*}
    \sigma\tau=\phi(\sigma)\phi(\tau) \hspace{2cm} \text{for } \sigma,\tau\in S.
\end{equation*}
Now, we may prove:
\begin{prop}\label{p:Walter}
    Let $n\in\mathbb{N}$ and let $\alpha\in\PBd_{n}$ with $\rank(\alpha)=0$. Then, the variant $\PBd^{\alpha}_{n}$ is the inflation of the rectangular band
    \begin{equation}
    \label{e:regPart}
        \Reg(\PBd^{\alpha}_{n})=\Pp^{\alpha}=\{\sigma\in\PBd_{n}:\rank(\sigma)=0\}
    \end{equation}
    along the map $\phi:\PBd_{n}\rightarrow\Pp^{\alpha}:\xi\mapsto \xi\alpha\xi$.
\end{prop}
\begin{proof}
    From the discussion following Hypothesis \ref{h:rank0} we have (\ref{e:regPart}) and it follows that $\Pp^{\alpha}$ is a subsemigroup of $\PBd^{\alpha}_{n}$. Furthermore, since $\rank(\alpha)=0$, it is easily seen that $\sigma\alpha\sigma=\sigma$ for any $\sigma\in\Pp^{\alpha}$, and $\alpha\tau\alpha=\alpha$ for any $\tau\in\PBd_{n}$.
    Thus, for $\sigma,\tau\in\PBd_{n}$ we have
    \begin{equation*}
        \phi(\sigma)\star_{\sigma}\phi(\tau)=
        (\sigma\alpha\sigma)\alpha(\tau\alpha\tau)
        =\sigma(\alpha\sigma\alpha\tau\alpha)\tau
        =\sigma\alpha\tau=\sigma\star_{\alpha}\tau.\qedhere
    \end{equation*}
\end{proof}

It is easily seen that any isomorphism of variants preserves the sizes of the pre-images under $\phi$ of the elements of $\Pp^{\alpha}$. Thus, if variants $\PBd^{\alpha}_{n}$ and $\PBd^{\beta}_{n}$ have different sets, counted with multiplicities, of sizes of pre-images under $\phi_{\alpha}:\xi\mapsto\xi\star_{\alpha}\xi$ and $\phi_{\beta}:\xi\mapsto \xi\star_{\beta}\xi$, respectively, of the set $\Pp^{\alpha}=\Pp^{\beta}$, they are non-isomorphic. However, the reverse implication does not hold. Namely, in Example \ref{ex:kraj}, both $\PBd^{\alpha}_{n}$ and $\PBd^{\beta}_{n}$ have the same set of sizes of the pre-images under $\phi_{\alpha}$ and $\phi_{\beta}$, respectively, $\{1,1,3,5\}$, but are proved to be non-isomorphic.

\section*{Acknowledgements}

This work was supported by the Serbian Ministry of Education, Science and
Technological Development through the Mathematical Institute of the Serbian
Academy of Sciences and Arts.
The author would like to thank the referee for their suggestions, which greatly improved the readability of the paper.
Additionally, the author thanks Prof.\ James East for his immensely helpful comments, and Prof.\ Volodymyr Mazorchuk, for suggesting Proposition \ref{p:Walter}.

\bibliography{refs1}
\bibliographystyle{plain}

\end{document}